\documentclass{lms}
\usepackage{amsmath,amssymb,amscd,amsfonts}
\usepackage{graphicx,subfigure}
\usepackage{hyperref}
\usepackage{dsfont}
\usepackage{color}
\usepackage{float}

\usepackage{anysize}
\usepackage{color}
\usepackage{paralist}

\usepackage[arrow,curve,matrix,tips,2cell]{xy}  
  \SelectTips{eu}{10} \UseTips
  \UseAllTwocells
  \usepackage{tikz,calc,etex}
  \usetikzlibrary{matrix, arrows, calc,backgrounds,fit,shapes}

\newtheorem{theorem}{Theorem}[section] 
\newtheorem{corollary}[theorem]{Corollary}

\newtheorem{claim}[theorem]{Claim}

\newnumbered{assertion}{Assertion}    
\newnumbered{conjecture}{Conjecture}  
\newnumbered{definition}{Definition}
\newnumbered{hypothesis}{Hypothesis}
\newnumbered{remark}{Remark}
\newnumbered{note}{Note}
\newnumbered{observation}{Observation}
\newnumbered{problem}{Problem}
\newnumbered{question}{Question}
\newnumbered{algorithm}{Algorithm}
\newnumbered{example}{Example}
\newunnumbered{notation}{Notation}
\newunnumbered{theoremp}{Theorem}

\newcommand{\rr}{\mathds{R}}

\newcommand{\Z}{\mathbb{Z}}
\newcommand{\F}{\mathbb{F}}

\newcommand{\K}{{\mathcal K}}
\newcommand{\X}{{\mathcal X}}

\def\rr{\mathds{R}}

\def\conf{\operatorname{Conf}}

\def\ind{\operatorname{Index}}

\newcommand\Sym{\mathfrak S}

\title{Thieves can make sandwiches}

\author{Pavle V. M. Blagojevi\'{c} and Pablo Sober\'on}

\dedication{Dedicated to Imre B\'ar\'any on the occasion of his 70th birthday}

\classno{52A99 , 52A37, 55S91 (primary), 05A18 (secondary).
 Please refer to {\tt http://www.ams.org/msc/} for a list of codes}

\extraline{Pavle V. M. Blagojevi\'{c} received funding from DFG via the Collaborative Research Center TRR~109 ``Discretization in Geometry and Dynamics'', and the grant ON 174008 of the Serbian Ministry of Education and Science.
This material is based upon work supported by the National Science Foundation under Grant No. DMS-1440140 while Pavle V. M. Blagojevi\'{c} was in residence at the Mathematical Sciences Research Institute in Berkeley, California, during the Fall of 2017.}

\begin{document}
\maketitle

\begin{abstract}
We prove a common generalization of the Ham Sandwich theorem and Alon's Necklace Splitting theorem.  Our main results show the existence of fair distributions of $m$ measures in $\rr^d$ among $r$ thieves using roughly $mr/d$ convex pieces, even in the cases when $m$ is larger than the dimension.  
	The main proof relies on a construction of a geometric realization of the topological join of two spaces of partitions of $\rr^d$ into convex parts, and the computation of the Fadell-Husseini ideal valued index of the resulting spaces.
\end{abstract}


\section{Introduction}

Measure partition problems are classical, significant and challenging questions of Discrete Geometry.
Typically easy to state, they hide a connection to various advanced methods of Algebraic Topology.  
In the usual setting, we are presented with a family of measures in a geometric space and a set of rules to partition the space into subsets, and we are asked if there is a partition of this type which splits each measure evenly.

\medskip
In this paper we consider convex partitions of the Euclidean space $\rr^d$.
More precisely, an ordered collection of $n$ closed subsets $\K=(K_1, \ldots, K_n)$ of $\rr^d$ is a {\em partition of $\rr^d$} if it is a covering $\rr^d=K_1\cup\cdots\cup K_n$, all the interiors $\mathrm{int}(K_1),\ldots,\mathrm{int}(K_n)$ are non-empty, and $\mathrm{int}(K_i)\cap \mathrm{int}(K_j)=\emptyset$ for all $1\leq i< j\leq n$.
A partition $\K=(K_1, \ldots, K_n)$ is a  {\em convex partition of $\rr^d$} if each element $K_i$ of the partition is convex.
Furthermore, for an integer $r\geq 1$ an {\em $r$-labeled (convex) partition of $\rr^d$} is an ordered pair $(\K,\ell)$ where $\K=(K_1, \ldots, K_n)$ is a (convex) partition of $\rr^d$, and $\ell\colon [n]\longrightarrow [r]$ is an arbitrary function.
We use the notation $[n]:=\{1,\ldots,n\}$.

\medskip
Throughout this paper all measures in $\rr^d$ are assumed to be probability measures that are absolutely continuous with respect to the Lebesgue measure.  
This in particular means that the overlapping boundary $\bigcup_{1\leq i<j\leq n} K_i\cap K_j$ of the elements of a partition always has measure zero.  

\medskip
The quintessential measure partitioning result is the Ham Sandwich theorem, which was conjectured by Steinhaus and proved subsequently by Banach in 1938 (consult for example \cite{BeyerZardecki}).
The Ham Sandwich theorem is one of the most widely known consequences of the Borsuk--Ulam theorem.

\begin{theoremp}[(Ham Sandwich theorem)]
Let $d\geq1$ be an integer.
For any collection of $d$ measures $\mu_1,\ldots,\mu_d$ in $\rr^d$, there exists an affine hyperplane $H$ that simultaneously splits them into halves.  Namely, we have
\[
\mu_i(H^+)=\mu_i(H^-)
\]
for all $1\leq i\leq d$, where $H^+$ and $H^-$ denote closed half-spaces determined by $H$.
\end{theoremp}

The reason for its name is an illustration where each of the measures is thought of as a different ingredient floating in $\rr^d$.  
The goal is to make two sandwiches with equal amount of each ingredient by cutting $\rr^d$ with a single hyperplane slice.  
Furthermore, if more people want their sandwich and they all fancy convex shapes, it has been shown by different groups of authors \cite{Soberon:2012kp} \cite{Karasev:2013gi} \cite{Blagojevic:2014ey} that \textit{for any collection of $d$ measures $\mu_1,\ldots,\mu_d$ in $\rr^d$ and any integer  $r\geq 1$ there is a convex partition $\K=(K_1, \ldots, K_r)$ of $\rr^d$ into $r$ parts that simultaneously split each measure into $r$ parts of equal size.  Namely,
\[
\mu_i(K_1)=\cdots=\mu_i(K_r)
\]
for all $1\leq i\leq d$. 
}

\medskip
The second classic and appealing measure partition result, and again a consequence of the Borsuk--Ulam theorem, is the following ``Necklace Splitting'' theorem of Hobby and Rice \cite{Hobby:1965bh}.

\begin{theoremp}[(Necklace Splitting theorem)]
Let $m\geq 1$ be an integer.
Then for any collection of $m$ measures $\mu_1, \ldots, \mu_m$ in $\rr$, there exists a $2$-labeled convex partition $(\K,\ell)$ of $\rr$ into a collection of $m+1$ intervals $\K=(K_1,\ldots, K_{m+1})$ with $\ell\colon [m+1]\longrightarrow [2]$ such that for every $1\leq j\leq m$:
\[
\mu_j\big(\bigcup_{i\in\ell^{-1}(1)}K_i\big)= \mu_j\big(\bigcup_{i\in\ell^{-1}(2)}K_i\big)=\tfrac{1}{2}.
\]
\end{theoremp}

The discrete version of this theorem, where the measures are the counting measures of finite sets of points, was first proved by Goldberg and West \cite{Goldberg:1985jr}, and then by Alon and West \cite{Alon:1985cy}. 
The common illustration of this theorem is as follows.  
Two thieves steal an open necklace with $m$ types of pearls, knowing that there is an even number of pearls of each kind.  
They will cut the necklace into pieces and distribute those among themselves, so that each receives half of each kind of pearl.  
The result above shows that this can always be achieved with $m$ cuts, regardless of the order of the pearls.   The function $\ell$ is simply telling us who gets each part. 
The version with arbitrary number $r$ of thieves was given by Alon in \cite{Alon:1987ur}, where $(r-1)m$ cuts are shown to be sufficient.  The number of cuts cannot be improved.  Extensions of this result with additional combinatorial conditions on the distribution of the necklace appear in \cite{asada2017fair}.

\medskip
Our main goal in this manuscript is to present a common generalization of the Ham Sandwich theorem and the Necklace Splitting theorem.  
In other words, given more than $d$ ingredients in $\rr^d$, we should be able to find a fair distribution among $r$ hungry persons if we are willing to split $\rr^d$ into more than $r$ parts.  
Alternatively, if $r$ thieves steal a high-dimensional necklace, they should be able to distribute it among themselves by splitting it into very few convex parts.

\medskip
High-dimensional versions of the Necklace Splitting theorem were given by de Longueville and \v{Z}ivaljevi\'c \cite{deLongueville:2006uo}, and by Karasev, Rold\'an-Pensado and Sober\'on \cite{Karasev:2016vd}. 
In \cite{deLongueville:2006uo}, the authors proved an analogous result for $r$ thieves and $m$ measures in $\rr^d$, where the partitions are made using $(r-1)m$ hyperplanes each of whose directions is fixed in advance and must be orthogonal to a vector of the canonical basis of $\rr^d$.  
The downside of this type of partitions is that there may be an extremely large number of pieces to distribute.

\begin{figure}[ht]
\centerline{\includegraphics[scale=0.8]{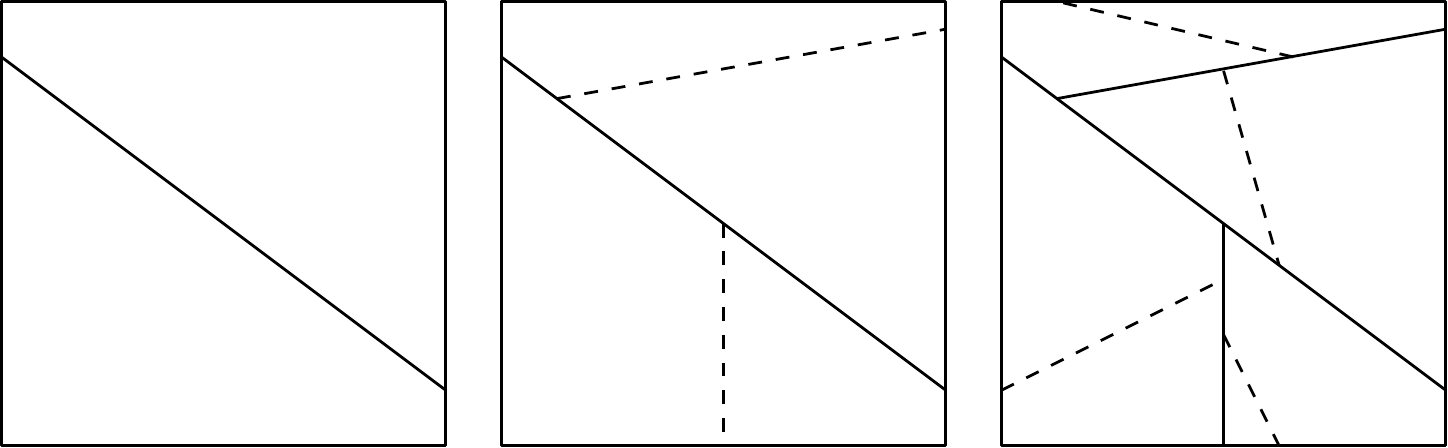}}
\caption{An iterated partition of the plane by successive hyperplane cuts}
\label{figure-1}
\end{figure}

\medskip
One way to address this issue is to consider iterated hyperplane partitions.  
A partition of $\rr^d$ into convex parts is an \textit{iterated hyperplane partition} if it can be made out of $\rr^d$ by successively partitioning each convex part, from a previous partition, with a hyperplane (that means each new hyperplane only cuts one of the existing convex parts, see Figure \ref{figure-1}).  
In \cite{Karasev:2016vd} the authors showed that for $r$ thieves and $m$ measures in $\rr^d$, there is a fair distribution of each measure among the thieves using an iterated hyperplane partition that has $(r-1)m$ hyperplane cuts, whose directions are fixed in advance, as long as $r$ is a prime power.  
This has the advantage that the total number of parts is $(r-1)m+1$.

\medskip
In both results, there is little to gain from the increasing dimension. 
This is a consequence of fixing the directions of the cutting hyperplanes.
Thus, it is natural to wonder what can be gained if the fixed directions restriction is disregarded.  In this situation we distinguish three different types of labeled partitions.  
The first type of partitions are labeled partitions of $\rr^d$ into $n$ convex parts without any additional requirements.  
For the second type of partitions we consider iterated convex partitions of $\rr^d$ that in the case when $r=2$ coincide with iterated hyperplane partitions.

Intuitively, these are formed by splitting $\rr^d$ using iterated hyperplane cuts, and then splitting each remaining region with a power diagram into $r$ parts.  Because of its complexity, the formal definitions of power diagrams and iterated convex partitions are postponed for the next section, Definition \ref{def : space of iterated partitions}.  Finally, the third type of partitions are those made by iterated hyperplane cuts. 

\begin{definition}
\label{def-M'}
Let $n\geq 1$, $r\geq 1$ and $d\geq 1$ be integers.
\begin{compactenum}[\rm (1)]
\item 	Let $M=M(n,r,d)$ be the largest integer such that for any collection of $M$ measures $\mu_1, \ldots, \mu_M$ in $\rr^d$ there exists an $r$-labeled convex partition $(\K,\ell)$ of $\rr^d$ into $n$ parts $\K=(K_1,\ldots, K_{n})$ with $\ell\colon [n]\longrightarrow [r]$ having the property that for all $1\leq j\leq M$ and all $1\leq s\leq r$:
\[
\mu_j\big(\bigcup_{i\in\ell^{-1}(s)}K_i\big)=\tfrac{1}{r}.
\] 
Every such labeled convex partition into $n$ parts is called a {\em fair distribution} between the thieves.

\item If $n$ is a multiple of $r$, let $M'=M'(n,r,d)$ be the largest integer such that for any collection of $M'$ measures $\mu_1, \ldots, \mu_{M'}$ in $\rr^d$ there exists an $r$-labeled iterated  partition $(\K,\ell)$ of $\rr^d$ into $n$ convex parts $\K=(K_1,\ldots, K_{n})$ with $\ell\colon [n]\longrightarrow [r]$ so that for all $1\leq j\leq M'$ and all $1\leq s\leq r$:
\[
\mu_j\big(\bigcup_{i\in\ell^{-1}(s)}K_i\big)=\tfrac{1}{r}.
\]
Every such $r$-labeled convex partition into $n$ parts is called a {\em fair iterated distribution} between the thieves.

\item Let $M''=M''(n,r,d)$ be the largest integer such that for any collection of $M''$ measures $\mu_1, \ldots, \mu_{M''}$ in $\rr^d$ there exists an $r$-labeled convex partition $(\K,\ell)$ of $\rr^d$ into $n$ parts $\K=(K_1,\ldots, K_{n})$ formed only by iterated hyperplane cuts with $\ell\colon [n]\longrightarrow [r]$ so that for all $1\leq j\leq M''$ and all $1\leq s\leq r$:
\[
\mu_j\big(\bigcup_{i\in\ell^{-1}(s)}K_i\big)=\tfrac{1}{r}.
\]
Every such $r$-labeled convex partition into $n$ parts is called a fair iterated distribution between the thieves by hyperplane cuts.
\end{compactenum}

\end{definition}

Some of the convex parts $K_i$ in the partition $\K$ can be empty.
For an example of a $2$-labeled convex partition formed by iterated hyperplane cuts see Figure \ref{figure-2}.

\begin{figure}[ht]
\centerline{\includegraphics[scale=1.2]{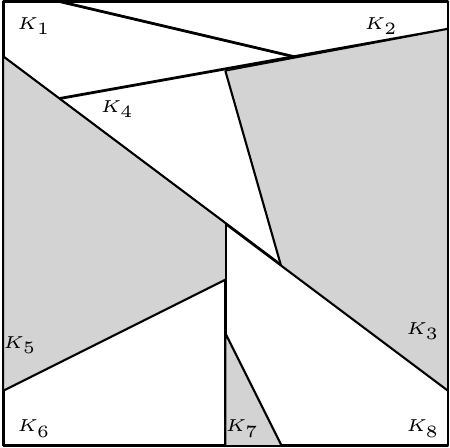}}
\caption{Convex $2$-labeled partition of $\rr^2$ by iterated hyperplane cuts with $\ell={{12345678}\choose {11212121}}$.}
\label{figure-2}
\end{figure}

For all integers $n\geq 1$, $r\geq 1$ and $d\geq 1$ it is clear that $ M'(n,r,d)\leq M(n,r,d)$ and $M''(n,r,d) \le M(n,r,d)$.  
The Ham Sandwich theorem is equivalent to the statement that $M''(2,2,d) = M'(2,2,d) = M(2,2,d) = d$, while the Necklace Splitting theorem for two thieves is equivalent to $M''(n+1,2,1) = M(n+1,2,1) = n$.  
The respective extensions of the Ham Sandwich and Necklace splitting theorems for distributions among more persons simply state that $M(r,r,d) = d$ and $M(n,r,1) = \lfloor \frac{n-1}{r-1}\rfloor$.

\medskip
In this paper we prove the following bound on the function $M'(n,r,d)$.

\begin{theorem}\label{theorem-main}
Let $d\geq1$ and $t\geq1$ be integers, and let $r\geq 2$ be a prime. 
Then,
\[
M'(rt,r,d)\geq \lceil \tfrac{td(r-1)+t}{r-1}-1 \rceil.
\] 
Moreover, this result is optimal for $M'(r,r,d)$ and $M'(n,2,1)$.
\end{theorem}

The labeled partitions we use to prove Theorem \ref{theorem-main} have additional property: 
From the $rt$ convex parts, every thief receives exactly $t$ of them.  
The result above implies that $M(r,r,d) = d$ for any $r$ using a standard factorization argument.  This factorization argument only works well if $t=1$ or $d=1$.

\medskip
For the case $r=2$, our results actually give iterated hyperplane partitions.  Since we also include results for the case when we use an odd number of parts, we state them separately.

\begin{theorem}\label{theorem-two-thieves}
Let $d\geq1$ be an integer.  
Then,
\[
\begin{array}{llll}
M''(2t,2,d)   &\geq  &	t(d+1)-1,  & \mbox{for }t\ge 1\\
M''(2t+1,2,d) &\geq  &   t(d+1) ,   & \mbox{for }t\ge 0.
\end{array}
\]
Moreover, this result is optimal for $M'(n,2,1), M'(3,2,d)$ and $M'(2,2,d)$.	
\end{theorem}

For $r=2$ Theorem \ref{theorem-two-thieves} says that $M'(n,2,d) \sim \left\lfloor\frac{n}{2}\right\rfloor(d+1)$, which can be seen as a common extension of the Ham Sandwich theorem and the Necklace Splitting theorem clearer.  
For larger values of $r$, the lower bounds we obtain for $M(n,d,r)$ are roughly $\frac{nd}{r}$.

\medskip
The rest of the paper is organized as follows.  
The configuration spaces of the partitions we use for the proof of Theorems \ref{theorem-main} and \ref{theorem-two-thieves} are introduced in Section \ref{section-spaces}. 
The lower bounds for Theorems \ref{theorem-main} and \ref{theorem-two-thieves} are established in Section \ref{section-proofs}.  
All our upper bounds for $M(n,r,d)$ and $M''(n,r,d)$ are showed in Section \ref{section-upperbounds}. 
The paper is conclude with a fewopen questions and remarks in Section \ref{section-remarks}.

\section{Configuration spaces of labeled convex partitions}

\label{section-spaces}

The proof of our main results relies on finding a well-behaved configuration space of iterated partitions.
This allows us to apply the configuration space / test map scheme and reduce the problem to the question about the non-existence of a particular equivariant map; for a classical introduction to this method see \cite{matousek2003using}.
In this section we introduce the relevant configuration spaces.
 
\medskip
First we recall the notion of oriented affine hyperplane including the hyperplanes at infinity.
Let $v\in S^{d-1}$ be a unit vector in $\rr^d$ and $a\in\rr\cup\{-\infty,+\infty\}$.
An oriented affine hyperplane determined by the pair $(v,a)$ is the pair $(H^+_{(v,a)},H^-_{(v,a)})$ of closed subsets of $\rr^d$ defined by:
\[
H^+_{(v,a)}:=
{
\begin{cases}
	\{ x\in\rr^d : \langle x,v\rangle\geq a\}, & a\in\rr,\\
	\emptyset, & a=+\infty,\\ 
	\rr^d, &  a=-\infty,\\
\end{cases}}
\quad\text{and}\quad
H^-_{(v,a)}:=
{
\begin{cases}
	\{ x\in\rr^d : \langle x,v\rangle\leq a\}, & a\in\rr,\\
	\rr^d, &  a=+\infty,\\ 
	\emptyset, & a=-\infty,\\
\end{cases}}
\]
For an example of an oriented affine hyperplane in a plane see Figure \ref{figure-3}.

\begin{figure}
\centerline{\includegraphics[scale=0.7]{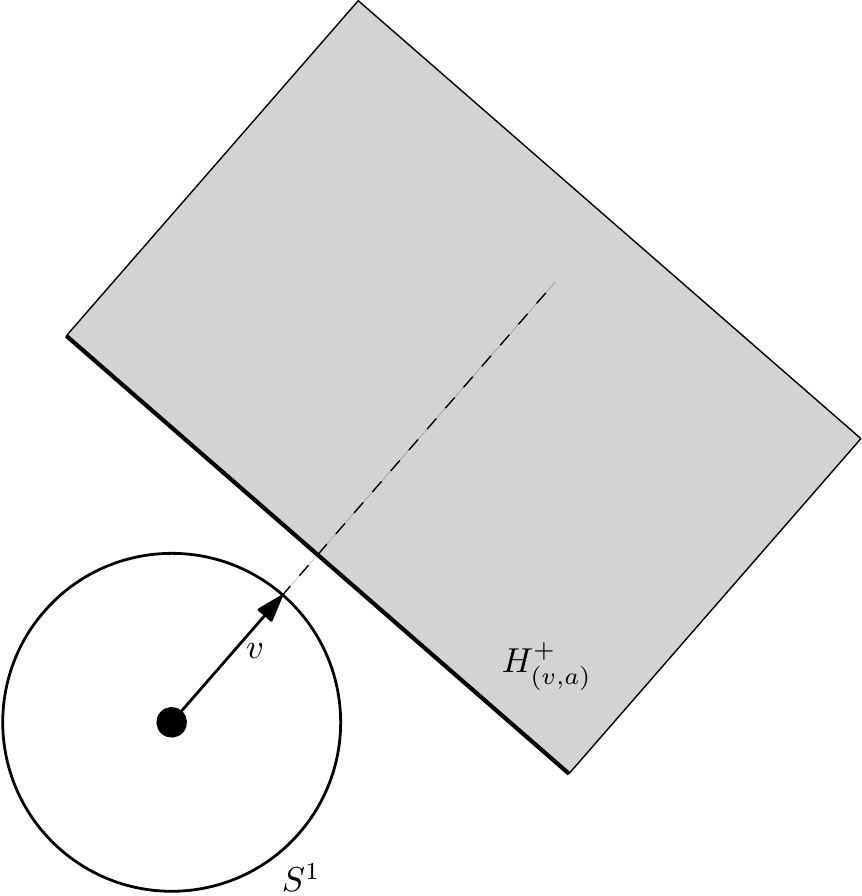}}
\caption{An illustration of an oriented affine hyperplane in the plane.}
\label{figure-3}
\end{figure}

\medskip
Next we introduce a family of binary operations on $r$-labeled convex partitions of $\rr^d$ that are indexed by oriented affine hyperplanes in $\rr^d$.

\begin{definition}
	Let $d\geq1$, $n\geq1$, $m\geq1$  and $r\geq1$ be integers, let $v\in S^{d-1}$ be a unit vector in $\rr^d$, and let $a\in\rr\cup\{-\infty,+\infty\}$.
	For any two $r$-labeled convex partitions 
\[
(\K,\ell)=((K_1, \ldots, K_n),[n]\longrightarrow [r]) 
\qquad\text{and}\qquad  
(\K',\ell')=((K_1', \ldots, K_m'),[m]\longrightarrow [r])
\]
and the oriented affine hyperplane $(H^+_{(v,a)},H^-_{(v,a)})$ we define the $r$-labeled convex partition $(\K'',\ell'')$ to be
\[
\K'':=(H^+_{(v,a)}\cap K_1,\ \ldots,H^+_{(v,a)}\cap K_n, \ H^-_{(v,a)}\cap K_1', \ \ldots, \ H^-_{(v,a)}\cap K_n')
\]
with $\ell''\colon [n+m]\longrightarrow [r]$ given by
\[
\ell''(k):=
\begin{cases}
	\ell(k), & 1\leq k\leq n,\\
	\ell'(k-n)& n+1\leq k\leq n+m.
\end{cases}
\]
\end{definition}

\noindent
We denote by $\K *_{(v,a)}\K'$ the labeled partition $\K''$ that was just defined, for an illustration see Figure \ref{figure-4}.

\begin{figure}
\centerline{\includegraphics[scale=0.8]{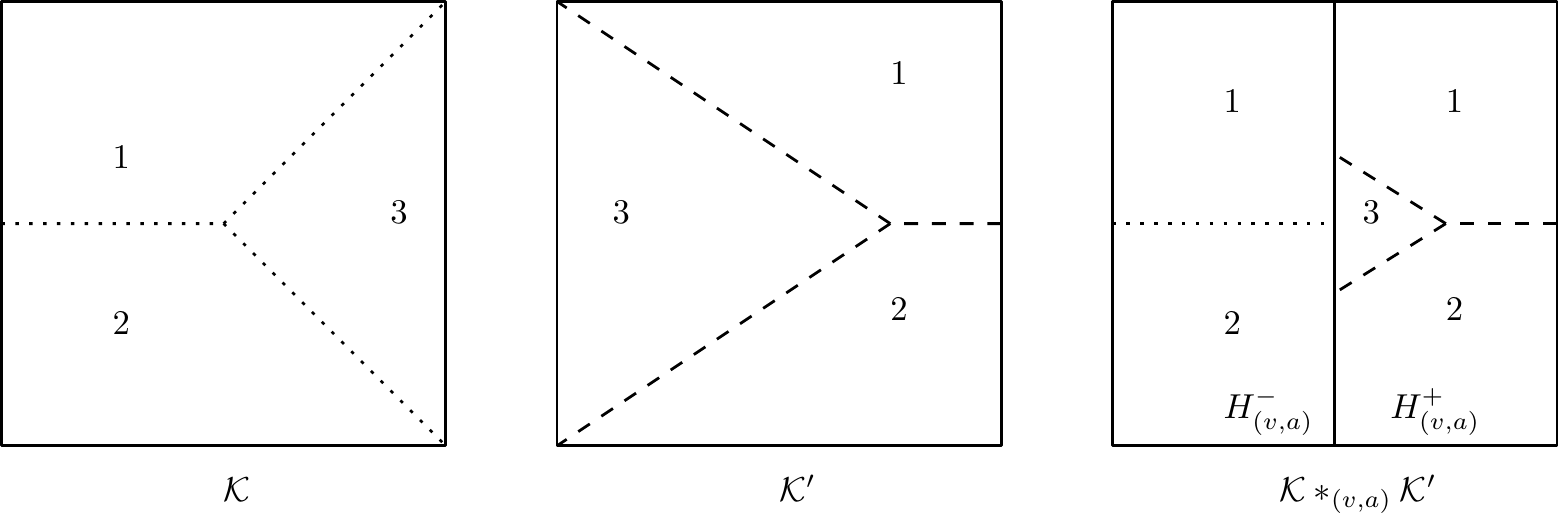}}
\caption{\small An illustration of of the operation $\K *_{(v,a)}\K'$.}
\label{figure-4}
\end{figure}

\medskip
The initial building block is the space of convex partitions of $\rr^d$ we use are power diagrams.  These are generalizations of Voronoi diagrams, defined below.  For a recent example of these partitions in the context of measure partitions see the work of Karasev \cite{Karasev2010}, and for a general survey see \cite{aurenhammer1991voronoi}.

\begin{definition}
	Let $d\geq1$ and $r\geq1$ be integers, and let $A_d$ be the space of all affine functionals $\rr^d\longrightarrow\rr$. 
	To every ordered $n$-tuple $(a_1,\ldots,a_r)$ of pairwise distinct elements of $A_d$ we associate the labeled convex partition $\K_{(a_1,\ldots,a_r)}:=((K_1,\ldots,K_r),\ell)$ where
	\[
	K_i:=\{x\in\rr^d : a_i(x)=\max\big\{a_1(x),\ldots,a_r(x)\}\big\},
	\]
	for all $1\leq i\leq r$, and $\ell\colon [r]\longrightarrow [r]$ is the identity map.
	The space $\X_{1,r}$ of all convex partitions $\K_{(a_1,\ldots,a_r)}$ parametrized by collections of $r$ pairwise distinct elements of $A_d$ can be identified with the classical configuration space $\conf(\rr^{d+1},r)$ of all $r$  pairwise distinct points in $\rr^{d+1}$, that is
	\[
	\X_{1,r}\cong \conf(\rr^{d+1},r):=\{(x_1,\ldots,x_r)\in (\rr^{d+1})^r : x_i\neq x_j \text{ for }i\neq j\}.
	\]
	Furthermore, the symmetric group $\Sym_r$ acts from the left on the space of partitions $\X_{1,r}$ as follows
	\[
	\pi\cdot ((K_1,\ldots,K_r),\ell) = ((K_{\pi(1)},\ldots,K_{\pi(r)}),\pi^{-1}\circ\ell),
	\]
	where $\pi\in\Sym_r$, and $\ell$ is the identity map.
\end{definition}

In the case $r=2$ notice that we are dealing with partitions determined by a single hyperplane; that means $\conf(\rr^{d+1},2) \simeq S^d$, the $d$-dimensional sphere in $\rr^{d+1}$.

\medskip
Now we use the operations $*_{(v,a)}$ between labeled partitions to introduce spaces of iterated labeled convex partitions for $r$ thieves starting with the space $\X_{1,r}$.

\begin{definition}
	\label{def : space of iterated partitions}
	Let $d\geq1$, $r\geq1$ and $t\geq 1$ be integers.
	A topological space $\X$ is a space of $t$-iterated $r$-labeled convex partitions if
	\begin{compactenum}[\rm (1)]
	\item $\X=\X_{1,r}$ when $t=1$, or if
	\item there exists $\X'$ a space of $t'$ iterated $r$-labeled convex partitions, $\X''$ a space of $t''$ iterated $r$-labeled convex partitions, and  unit vector $v\in S^{d-1}$ such that $t=t'+t''$ and
	\[
	\X=\{\K'*_{(v,a)}\K'' : \K'\in\X',\K''\in\X'', a\in\rr\cup\{-\infty,+\infty\}\}.
	\]
	\end{compactenum}
	An $r$-labeled convex partition $(\K,\ell)$ is an iterated $r$-labeled convex partition if for some integer $t\geq1$ it belongs to a space of $t$-iterated $r$-labeled convex partitions.
		\end{definition}
		
\begin{figure}
\centerline{\includegraphics[scale=0.7]{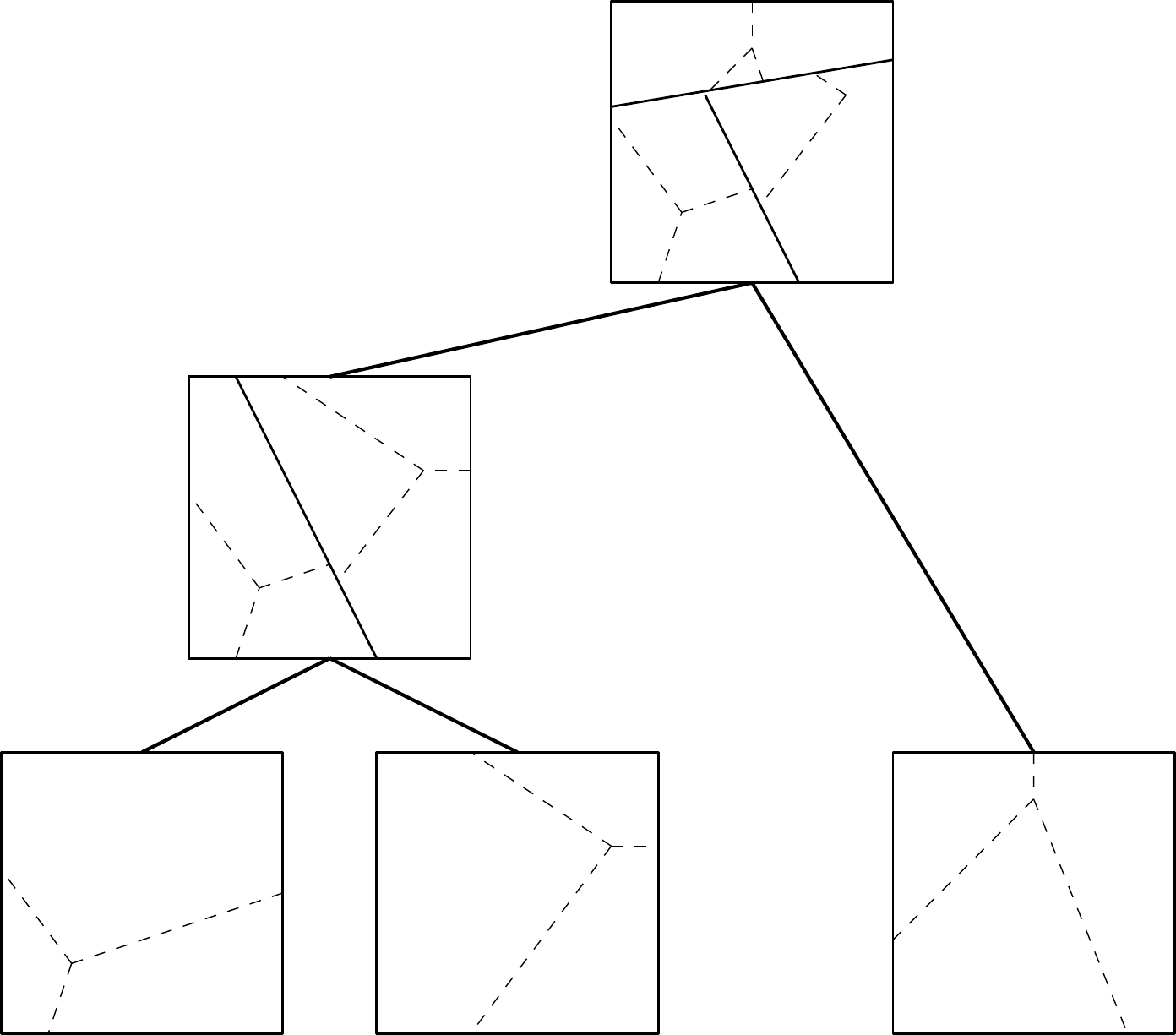}}
\caption{\small An example of a construction of a $3$-iterated partition from a binary tree.  Each lower block with dotted lines represents a partition given by $\conf(\rr^2,3)$.  The continuous lines are the hyperplanes with fixed directions used to combine the lower blocks.  In the general construction, we can prescribe the partition tree and the direction of the hyperplanes involved in the construction.}
\label{figure-tree}
\end{figure}		
 
An intuitive way to think about the partitions is that we are iteratively making hyperplane cuts to existing regions, and then we cut every remaining set with a power diagram.  
For an illustration of how an iterated partition is build from a binary three and a choice of direction vectors see Figure \ref{figure-tree}.
According to the definition of the operations $*_{(v,a)}$ it follows that for any given space $\X$ of $t$-iterated $r$-labeled partitions there is a natural surjective map from the iterated joins $\conf(\rr^{d+1},r)^{* t}$ onto $\X$.
If $\X$ is a space of $t$-iterated $r$-labeled convex partitions then every partition $(\K,\ell)\in\X$, where $\ell\colon [tr]\longrightarrow [r]$, has at most $tr$ non-empty convex pieces. 
Furthermore, any space of $t$-iterated $r$-labeled convex partitions has the left diagonal $\Sym_r$ action induced from the action on $\X_{1,r}$.

\section{Proof of Theorem \ref{theorem-main} and Theorem \ref{theorem-two-thieves}}\label{section-proofs}
Let $d\geq1$, $t\geq1$ and $M\geq1$ be integers, let $r\geq 2$ be a prime, and let  $\mu_1, \ldots, \mu_M$ be measures in $\rr^d$.
Choose an arbitrary $t$-iterated $r$-labeled space of convex partitions $\X$.
As we have seen $\X$ can be identified with the $t$-fold join $\conf(\rr^{d+1},r)^{*t}$ of the classical configuration space.
A typical element $(\K,\ell)$ of $\X$ is a $t$-iterated $r$-labeled convex partition where $\ell\colon [rt]\longrightarrow [r]$.
Consider the continuous map $\Phi\colon \X \longrightarrow W_r^{\oplus M}$ defined by
\begin{multline*}
(\K,\ell)=((K_1,\ldots,K_{rt}),\ell\colon [rt]\longrightarrow [r])\longmapsto  \\
\left(
\sum_{j\in\ell^{-1}(1)}	\mu_i(K_j)-\tfrac1r\Big(\sum_{j=1}^{rt}\mu_i(K_j)\Big)
,\ldots,
\sum_{j\in\ell^{-1}(r)}	\mu_i(K_j)-\tfrac1r\Big(\sum_{j=1}^{rt}\mu_i(K_j)\Big)
\right)_{i=1,\ldots,M},
\end{multline*}
where $W_r:=\{(y_1,\ldots,y_r)\in\rr^r : x_1+\cdots+x_r=0\}$.
The space of $t$-iterated $r$-labeled convex partitions $\X$ is equipped with an action of the symmetric group $\Sym_r$.
The vector space $W_r\subseteq\rr^r$ can be seen as a real $\Sym_r$-representation via the action given by $\pi\cdot (y_1,\ldots,y_r)=(y_{\pi^{-1}(1)},\ldots,y_{\pi^{-1}(r)})$ for $\pi\in\Sym_r$ and $(y_1,\ldots,y_r)\in W_r$.
Then the direct power $W_r^{\oplus M}$ is endowed with the diagonal $\Sym_r$-action.
With the introduced actions the map $\Phi$ is an $\Sym_r$-equivariant map.
Furthermore, if the image $\mathrm{im}(\Phi)$ of the map $\Phi$ contains $0\in W_r^{\oplus M}$ then any partition in $\Phi^{-1}(0)$ is a witness that a fair iterated distribution between the thieves is possible, and consequently that 
\[
M(rt,r,d)\geq M.
\]

\medskip
Let us assume that $0\notin \mathrm{im}(\Phi)$, meaning that there are no fair iterated distribution between the thieves. 
Then the $\Sym_r$-equivariant map $\Phi$ factors as follows
\[
\xymatrix{
\conf(\rr^{d+1},r)^{*t}\ar[rr]^-{\Phi}\ar[dr]_{\Psi} &   &  W_r^{\oplus M}\\
&S(W_r^{\oplus M})\ar[ur]_{i}&
}
\]
where $\Psi\colon\conf(\rr^{d+1},r)^{*t}\longrightarrow S(W_r^{\oplus M})$ is the $\Sym_r$-equivariant map obtained by composing with the radial retraction to the unit sphere, and $i\colon S(W_r^{\oplus M})\longrightarrow W_r^{\oplus M}$ is the inclusion.
Thus, in order to conclude the proof Theorem \ref{theorem-main} we need to show that for $M\leq  \lceil \tfrac{td(r-1)+t}{r-1}-1 \rceil$ there is no $\Sym_r$-equivariant map $\conf(\rr^{d+1},r)^{*t}\longrightarrow S(W_r^{\oplus M})$.
More precisely we will prove that there cannot be any $\Z/r$-equivariant map $\conf(\rr^{d+1},r)^{*t}\longrightarrow S(W_r^{\oplus M})$ where $\Z/r$ is the subgroup of $\Sym_r$ generated by the cyclic permutation $(12\ldots r)$.

\medskip
At this point, we are ready to prove Theorem \ref{theorem-two-thieves} using only the Borsuk--Ulam theorem.

\begin{proof}[Proof of Theorem \ref{theorem-two-thieves}]
	If $r=2$ the space of $t$-iterated $2$-labeled partitions is formed only by iterated hyperplane partitions.  
	Let $M = t(d+1)-1$.  
	Then
	\[
	\conf(\rr^{d+1},2)^{*t} \simeq (S^d)^{*t} \cong ((\Z/2)^{*(d+1)})^{*t} \cong (\Z/2)^{*(d+1)t} \cong S^{t(d+1)-1}=S^M,
	\]  
	endowed with a free $\Z/2$-action.
	Furthermore, $W_2 = \rr$, so $S(W_2^{\oplus M}) \cong S^{M-1}$, also equipped with antipodal action.  
	The fact that there is no $\Z/2$-equivariant map  $\Psi \colon S^M \longrightarrow S^{M-1}$ is the content of the Borsuk--Ulam theorem.
	
	\smallskip
	For the second part of the theorem, we need a slight modification of our partitions.  
	Let $M = t(d+1)$, and let $\mathcal{Y}$ be the $2$-labeled empty partition.  
	That is, the partition of $\rr^d$ into just one set, equipped with a function $\ell:[1]\longrightarrow [2]$.  
	Let $\X$ be a $t$-iterated $2$-labeled partition of $\rr^d$, and $v$ a unit vector in $\rr^d$.  
	The space of partitions we use is $\mathcal{Z} = \{\K *_{(v,a)}\K': \K \in \mathcal{Y}, \K' \in \X, a \in \rr \cup \{-\infty,+\infty \}\}$.  
	Any partition of $\mathcal{Z}$ is a partition of $\rr^d$ into at most $2t+1$ parts using iterated hyperplane partitions.  We can parametrize $\mathcal{Z}$ by $\mathcal{Y} * \X$, and in turn this can be parametrized by $(\Z/2) * ((\Z/2)^{*(d+1)})^{*t} = (\Z/2)^{*(M+1)} \cong S^{M}$.  
	As before, the non-existence of a $\Z/2$-equivariant map $\Psi\colon S^M \to S^{M-1}$ concludes the proof.
\end{proof}

\medskip
For all primes $r$, including $r=2$, the non-existence of a $\Z/r$-equivariant map
\begin{equation}
	\label{eq : eq-map}
	\conf(\rr^{d+1},r)^{*t}\longrightarrow S(W_r^{\oplus M})
\end{equation}
can be established using the ideal valued index theory of Fadell and Husseini \cite{Fadell1988}.
First, we briefly recall the notion of the Fadell--Husseini ideal valued index and necessary properties in the generality we use them in our proof.

\medskip
Let $G$ be a finite group, and let $X$ be a $G$-space.
The {\em Fadell--Husseini index} of the space $X$ with respect to the group $G$ and the coefficients in the fields $\F$ is the  kernel of the following map in cohomology
\begin{eqnarray*}
	\ind_G (X;\F) &:=& 
	\ker\big( H^*(\mathrm{E}G\times_G\mathrm{pt};\F)\longrightarrow H^*(\mathrm{E}G\times_G X;\F)\big)\\
	&=& \ker\big( H^*(\mathrm{B}G;\F)\longrightarrow H^*(\mathrm{E}G\times_G X;\F)\big)\\
	&=& \ker\big( H^*(G;\F)\longrightarrow H^*(\mathrm{E}G\times_G X;\F)\big)\subseteq H^*(G;\F),
\end{eqnarray*}
that is induced by the $G$-equivariant projection $\pi_X\colon X\longrightarrow\mathrm{pt}$ that further on induces the continuous map $\pi_X\times_G\mathrm{id}\colon \mathrm{E}G\times_G X\longrightarrow \mathrm{E}G\times_G\mathrm{pt}$.
Here $\mathrm{pt}$ denotes the point equipped with the trivial $G$-action.
Silently we assume natural isomorphisms $ H^*(\mathrm{E}G\times_G\mathrm{pt};\F)\cong H^*(\mathrm{B}G;\F)\cong H^*(G;\F)$. 
The key feature of the Fadell--Husseini index, that allow us to answer questions about the existence of $G$-equivariant maps, is the following monotonicity property \cite[p.\,74]{Fadell1988}: If $X$ and $Y$ are $G$-spaces, and if $f\colon X\longrightarrow Y$ is a continuous $G$-equivariant map, then
\[
\ind_G (X;\F)\supseteq \ind_G (Y;\F).
\]
For our proof $G=\Z/r$ and we fix the notation of the cohomology of the cyclic group $\Z/r$ with coefficients in the field $\F_r$ as follows:
\[
H^*(\Z/r;\F_r)=
\begin{cases}
	\F_r[x], & \text{for }r=2, \text{ where }\deg(x)=1,\\
	\F_r[x]\otimes \Lambda[y], & \text{for }r\geq3, \text{ where }\deg(x)=2\text{ and }\deg(y)=1.
\end{cases}
\]
Here $\Lambda[\,\cdot\,]$ denotes the exterior algebra.

\medskip
Now, in order to prove the non-existence of the $\Z/r$-equivariant map \eqref{eq : eq-map} we will show that
\begin{compactenum}[\rm \qquad(a)]
\item 
$\ind_{\Z/r} (\conf(\rr^{d+1},r)^{*t};\F_r)=H^{\geq t(d(r-1)+1)}(\Z/r;\F_r)$, and
\item
$\ind_{\Z/r} (S(W_r^{\oplus M});\F_r)=H^{\geq M(r-1)}(\Z/r;\F_r)$.
\end{compactenum}
From the assumption of the theorem $M\leq  \lceil \tfrac{td(r-1)+t}{r-1}-1 \rceil$ we have that 
\[
M(r-1)\leq (r-1)  \lceil \tfrac{td(r-1)+t}{r-1}-1 \rceil< t(d(r-1)+1).
\]
Thus, for $r=2$ we have
\[
x^{M(r-1)}\in \ind_{\Z/r} (S(W_r^{\oplus M});\F_r){\setminus}\ind_{\Z/r} (\conf(\rr^{d+1},r)^{*t};\F_r),
\]
while for $r\geq3$ we obtain
\[
x^{M(r-1)/2}\in \ind_{\Z/r} (S(W_r^{\oplus M});\F_r){\setminus}\ind_{\Z/r} (\conf(\rr^{d+1},r)^{*t};\F_r),
\]
implying that
\[
\ind_{\Z/r} (\conf(\rr^{d+1},r)^{*t};\F_r)
\not\supseteq
\ind_{\Z/r} (S(W_r^{\oplus M});\F_r).
\]
Consequently, the $\Z/r$-equivariant map \eqref{eq : eq-map} cannot exist, and the proof of the theorem is concluded.
Hence, it remains to verify the index evaluations and the proof of the theorem is complete.

\subsubsection*{Evaluation of $\ind_{\Z/r} (\conf(\rr^{d+1},r)^{*t};\F_r)$:}
In the case when $t=1$ the claim that 
\[
\ind_{\Z/r} (\conf(\rr^{d+1},r);\F_r)=H^{\geq d(r-1)+1}(\Z/r;\F_r)
\]
is a content of \cite[Thm.\,6.1]{Blagojevic2012}.
This result can also be deduced from the Vanishing Theorem of Frederick~R.~Cohen~\cite[Theorem~8.2, page~268]{Cohen1976LNM533}.

For $t\geq 2$ we compute the index from the Serre spectral sequence associated to the fibration
\begin{equation}
	\label{eq-ss}
	\xymatrix{
\conf(\rr^{d+1},r)^{*t}\ar[r]   &  \mathrm{E}\Z/r\times_{\Z/r}\conf(\rr^{d+1},r)^{*t}\ar[r] &  \mathrm{B}\Z/r.
}
\end{equation}
The $E_2$-term of the spectral sequence is of the form
\[
E^{p,q}_2=H^p(\mathrm{B}\Z/r;\mathcal{H}^q(\conf(\rr^{d+1},r)^{*t};\F_r))\cong H^p(\Z/r;H^q(\conf(\rr^{d+1},r)^{*t};\F_r)).
\]
Here $\mathcal{H}(\,\cdot\,)$ denotes the local coefficients determined by the action of $\pi_1(\mathrm{B}\Z/r)\cong\Z/r$ on the cohomology $H^q(\conf(\rr^{d+1},r)^{*t};\F_r)$.
Recall that 
\[
\ind_{\Z/r} (\conf(\rr^{d+1},r)^{*t};\F_r)\cong \ker (E^{*,0}_2\longrightarrow E^{*,0}_{\infty}).
\]
In order to proceed with the computation of the spectral sequence we need to understand the coefficients, these are cohomologies $H^*(\conf(\rr^{d+1},r)^{*t};\F_r)$, as $\F_r[\Z/r]$-modules.
For that we use the K\"unneth theorem for joins and have that
\[
\widetilde{H}^{i+t-1}(\conf(\rr^{d+1},r)^{*t};\F_r)\cong
\sum_{a_1+\cdots+a_t=i}\widetilde{H}^{a_1}(\conf(\rr^{d+1},r);\F_r)\otimes\cdots\otimes \widetilde{H}^{a_t}(\conf(\rr^{d+1},r);\F_r),
\]
where the action of $\Z/r$ on the tensor product is the diagonal action.
The cohomology of the configuration space $H^*(\conf(\rr^{d+1},r);\F_r)$, as a $\F_r[\Z/r]$-module, was described in \cite[Proof of Thm.\,8.5]{Cohen1976LNM533} and \cite[Thm.\,3.1, Cor.\,6.2]{Blagojevic2012}.
In summary,
\[\
H^q(\conf(\rr^{d+1},r);\F_r)={\small
\begin{cases}
	\F_r, & \text{for }q=0,\\
	\F_r[\Z/r]^{a_q}, & \text{for }q=dj\text{ where }1\leq j\leq r-2,  \\
	\F_r[\Z/r]^{\tfrac{(r-1)!-r+1}r}\oplus\big(\F_r[\Z/r]/\F_r \big),  & \text{for }q=d(r-1),\\
	0, & \text{otherwise},
\end{cases}}
\]
for some integers $a_q\geq 1$.
Simply, the reduced cohomology $\widetilde{H}^q(\conf(\rr^{d+1},r);\F_r)$ is a free $\F_r[\Z/r]$-module if and only if $q\neq d(r-1)$.
Then, according to \cite[Lem.\,VI.11.7]{Hilton1997}, we have that the reduced cohomology of the join
$\widetilde{H}^{q}(\conf(\rr^{d+1},r)^{*t};\F_r)$ is a free $\F_r[\Z/r]$-module if and only if $q\neq t(d(r-1)+1)-1$.
\begin{figure}
\centerline{\includegraphics[scale=0.78]{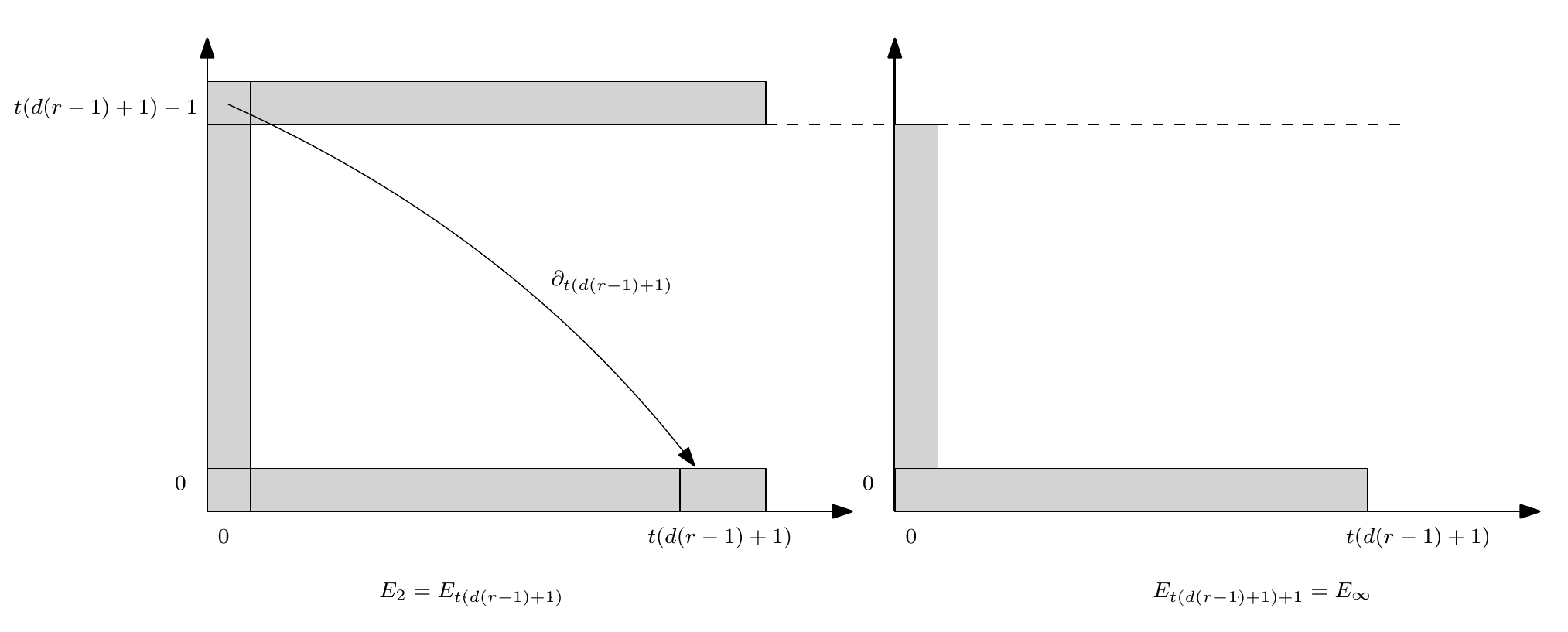}}
\caption{\small The Serre spectral sequence associated to the fibration \eqref{eq-ss}.}
\label{figure-ss}
\end{figure}
Since, $H^j(\Z/r;S)=0$ for all $j\geq1$ when $S$ is a free (projective) module we have that   
\[
E^{p,q}_2=
{\small
\begin{cases}
	H^p(\Z/r;\F_r), & \text{for }q=0,\\
	H^{q}(\conf(\rr^{d+1},r)^{*t};\F_r)^{\Z/r}, & \text{for }p=0,  \\
	\neq 0,  & \text{for }p\geq 0 \text{ and }q=t(d(r-1)+1)-1,\\
	0, & \text{otherwise}.
\end{cases}}
\]
Since the multiplication by $x$ in the cohomology of the group $\Z/r$ is an isomorphism, and 
$E^{p,q}_2=0$ for $p\geq 1$ and $1\leq q\leq  t(d(r-1)+1)-2$ we have that all the differentials $\partial_2,\ldots,\partial_{t(d(r-1)+1)-1}$ have to vanish, and so $E_2^{p,q}\cong\cdots\cong E_{t(d(r-1)+1)}^{p,q}$, see the illustration in Figure \ref{figure-ss}.
The only possible non-zero differential is $\partial_{t(d(r-1)+1)}$.
This means that 
\[
E_2^{p,0}\cong\cdots\cong E_{t(d(r-1)+1)}^{p,0}\cong E_{t(d(r-1)+1)+1}^{p,0}\cong E_{\infty}^{p,0}
\]
for $0\leq p\leq t(d(r-1)+1)-1$.
Consequently,
\[
\ind_{\Z/r} (\conf(\rr^{d+1},r)^{*t};\F_r)\subseteq H^{\geq t(d(r-1)+1)}(\Z/r;\F_r).
\]

\medskip
On the other hand $\conf(\rr^{d+1},r)$ and $\conf(\rr^{d+1},r)^{*t}$ are free $\Z/r$-spaces and therefore
\begin{equation}
\label{eq : hotopyeq}
	 \mathrm{E}\Z/r\times_{\Z/r}\conf(\rr^{d+1},r)^{*t} \simeq \conf(\rr^{d+1},r)^{*t}/(\Z/r).
\end{equation}
There exists an $\Sym_r$ simplicial complex $\mathfrak{C}(d+1,r)$ of dimension $d(r-1)$ that is $\Sym_r$-homotopy equivalent to the configuration space $\conf(\rr^{d+1},r)$; it was obtained in the construction of a particular $\Sym_r$-CW model for the configuration space in \cite{Blagojevic:2014ey}.
Consequently, its $t$-fold join $\mathfrak{C}(d+1,r)^{*t}$ is a $(t(d(r-1)+1)-1)$-dimensional $\Sym_r$-CW complex $\Sym_r$-homotopy equivalent to the join $\conf(\rr^{d+1},r)^{*t}$.
The homotopy equivalence \eqref{eq : hotopyeq} yields that
\begin{multline*}
	 H^i(\mathrm{E}\Z/r\times_{\Z/r}\conf(\rr^{d+1},r)^{*t};\F_r)\cong H^i(\conf(\rr^{d+1},r)^{*t}/(\Z/r);\F_r)\cong \\
	 H^i(\mathfrak{C}(d+1,r)^{*t}/(\Z/r);\F_r)=0
\end{multline*}
for all $i\geq t(d(r-1)+1)$.
In particular, since $\dim(\mathfrak{C}(d+1,r)^{*t}/(\Z/r))=t(d(r-1)+1)-1$, this means that $E^{p,0}_{\infty}=0$ for all $p\geq t(d(r-1)+1)$, implying the inclusion
\[
\ind_{\Z/r} (\conf(\rr^{d+1},r)^{*t};\F_r)\supseteq H^{\geq t(d(r-1)+1)}(\Z/r;\F_r).
\]
Thus we have verified that
\[
\ind_{\Z/r} (\conf(\rr^{d+1},r)^{*t};\F_r)= H^{\geq t(d(r-1)+1)}(\Z/r;\F_r).
\]

\subsubsection*{Evaluation of $\ind_{\Z/r} (S(W_r^{\oplus M});\F_r)$:}
Since $r$ is a prime the sphere $S(W_r^{\oplus M})$ is a free $\Z/r$-space.
Thus there is a homotopy equivalence 
\[
  \mathrm{E}\Z/r\times_{\Z/r} S(W_r^{\oplus M}) \simeq S(W_r^{\oplus M})/(\Z/r).
\]
Consequently, for all $i\geq \dim S(W_r^{\oplus M})+1=M(r-1)$ the cohomology has to vanish, that is 
\[
H^{i}(\mathrm{E}\Z/r\times_{\Z/r} S(W_r^{\oplus M});\F_r)\cong H^{i}(S(W_r^{\oplus M})/(\Z/r);\F_r)=0.
\]
Therefore,
\[
\ind_{\Z/r} (S(W_r^{\oplus M});\F_r)\supseteq H^{\geq M(r-1)}(\Z/r;\F_r).
\]
On the other hand the sphere $S(W_r^{\oplus M})$ is an $(M(r-1)-1)$-dimensional and $(M(r-1)-2)$-connected space.
Hence, the Serre spectral sequence associated to the fibration
\[
\xymatrix{
S(W_r^{\oplus M})\ar[r]   &  \mathrm{E}\Z/r\times_{\Z/r}S(W_r^{\oplus M})\ar[r] &  \mathrm{B}\Z/r
}
\]
yields that the induces map
$
H^i(\mathrm{E}\Z/r\times_{\Z/r}\mathrm{pt};\F_r)\longrightarrow H^i(\mathrm{E}\Z/r\times_{\Z/r} S(W_r^{\oplus M});\F_r)
$
is an injection for all $0\leq i\leq M(r-1)-1$.
Thus,
\[\ind_{\Z/r} (S(W_r^{\oplus M});\F_r)=H^{\geq M(r-1)}(\Z/r;\F_r),\]
as we have claimed.

\section{Upper bounds and additional examples}\label{section-upperbounds}
In this section we present upper bounds for the functions $M''(n,r,d)$ and $M(n,r,d)$.

\begin{claim}
Let $d\geq1$, $n\geq1$ and $r\geq2$ be integers.
Then 
$
M''(n,r,d) \le \left\lfloor \frac{d(n-1)}{r-1}\right\rfloor
$.
\end{claim}

 \begin{proof}
If we are allowed to use iterated hyperplane partitions of $\rr^d$ into $n$ parts, we are using $n-1$ hyperplanes. Suppose each measure is concentrated near a point and those points are in general position.  Then, each of the $n-1$ subdividing hyperplanes can cut at most $d$ measures.  However, each measure needs to be cut at least $r-1$ times in order to be distributed among $r$ thieves. 
	Consequently, \[d \cdot (n-1) \ge (r-1)M''(n,r,d).\vspace{-10pt}\]
 \end{proof}

\noindent
For the case $r=2$ this is similar to our bounds in Theorem \ref{theorem-two-thieves} except for a factor of two.

\begin{figure}
\centerline{\includegraphics[scale=0.80]{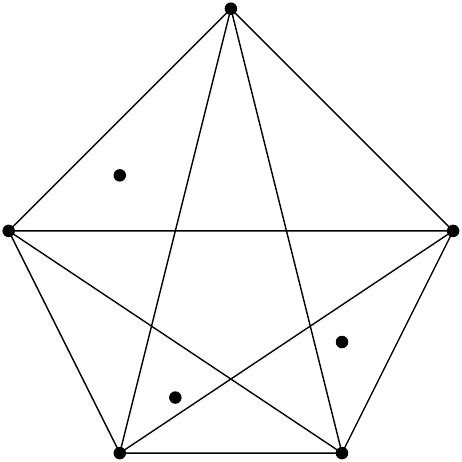}}
\caption{Set of eight points showing that $M(5,2,2) < 8$.}
\label{figure-extra}
\end{figure}

\medskip
The upper bounds obtained by this argument are not optimal.  
For example, let us consider the case when $r=2$ and $n=3$. 
The degree-counting argument above shows that $M' \le 2d$, while the following claim improves this significantly.

\begin{claim}
\label{claim-01}
Let $d\geq 1$ and $r\geq 2$ be integers. 
Then $M(2r-1,r,d) \le d+1$.
\end{claim}

\begin{proof}
It suffices to exibit a collection of $d+2$ measures that cannot be split among $r$ thieves using $2r-1$ parts.  Consider $d+1$ measures concentrated each near some vertex of a non-degenerate simplex in $\rr^d$, and one measure concentrated near a point in the interior of the simplex.

\medskip
If a fair distribution among $r$ thieves exists using a partition into $2r-1$ parts, there is a thief that received exactly one part $K$.  Thus, the convex set $K$ must have points from each of the first $d+1$ measures.  This would force the lucky thief to have all of the last measure, which would not be acceptable by the rest.
\end{proof}

\begin{corollary}
Let $d\geq 1$ be an integer. 
Then $M''(3,2,d) = d+1$.
\end{corollary}

\begin{proof}
From Theorem \ref{theorem-two-thieves} and Claim \ref{claim-01} we deduce that
$d+1 \le M''(3,2,d) \le M(3,2,d) \le d+1$.
\end{proof}

This upper bound shows that the behavior of $M(n,r,d)$ indeed has sudden sharp increases of values.  
More precisely, we have $M(r,r,d) = d$, $M(2r-1, r, d) \le d+1$ and $M(2r,r,d) \ge 2d$ (for $r$ a prime number).  Thus, $M(n,r,d)$ can exhibit sharp increases as $n$ grows.

\begin{claim}
$M(5,2,2) \le 7$.
\end{claim}

\begin{proof}
Consider a set of eight measures concentrated near the vertices of the set described in Figure \ref{figure-extra}.  We distinguish the five external measures (near the vertices of the circumscribing pentagon) from the other three internal measures.  If we divide $\rr^2$ into five convex pieces to be distributed among two thieves, one of them must receive at most two pieces.  Of these two pieces, one must interesect at least three of the external measures.  However, this means that the piece contains all of one of the internal measures, which negates a fair distribution.
\end{proof}

\section{Open questions and remarks}\label{section-remarks}

One of the results mentioned in the introduction is that $M(r,r,d) = d$, established in  \cite{Soberon:2012kp} \cite{Karasev:2013gi} \cite{Blagojevic:2014ey}.  
However, it is not clear if this can be said for iterated hyperplane partitions.  One may immediately notice that $M''(r,r,d)=1$ if $r$ is odd.  To see this, consider two measures distributed uniformly in two concentric spheres of distinct radii.  If we find a fair distribution into $r$ equal pieces with nested hyperplane cuts, the first cut must simultaneously cut both measures into two sets $A_1$, $A_2$ of sizes $\frac{k}{r}$ and $\frac{r-k}{r}$ for some positive integer $k <r$.  
This determines the distance from this hyperplane from the center of the spheres, which must be positive as $\frac{k}{r} \neq \frac{1}{2}$.  However, each measure would require a different positive distance, which makes this impossible.  For $r$ even, the value of $M''(r,r,d)$ is still open.

\begin{problem}
Given $d$ measures in $\rr^d$, and $r$ an even integer is it possible to find a partition of $\rr^d$ into $r$ parts using an iterated hyperplane partition so that each of the resulting sets has the same size in each measure?  In other words, is it true that $M''(r,r,d) = d$ for $r$ even?
\end{problem}

One of the minimal open cases regards partitions using four nested hyperplanes.  We have the bounds $2d+1 \le M''(4,2,d) \le 3d$.  Finding the rate of growth of $M''(4,2,d)$ is an interesting problem.  In particular, in the plane we know that $M''(4,2,2)$ is either five or six, and determining the exact value is open.  For more than two thieves, the first open case for which Theorem \ref{theorem-main} applies gives the bound $M(6,3,2) \ge 4$.  Understanding these first cases would shed light on the general behavior of $M(n,r,d)$.

\begin{figure}[ht]
\centerline{\includegraphics[scale=0.85]{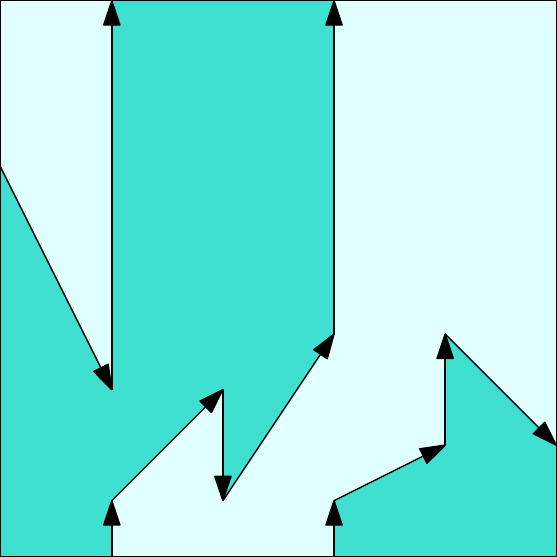}}
\caption{An example of a path $L$ with $t = 4$ that goes twice through infinity.}
\label{figure-path}
\end{figure}

The result above is trivial if $r=2^k$ for some $k$, as it suffices to make an iterated use of the Ham Sandwich theorem.  
A factorization argument also shows that if $M''(a,a,d) = d$ and $M''(b,b,d)=d$, then $M'(ab,ab,d)=d$.  Thus, it suffices to solve the problem above for $r=2q$ with $q$ an odd integer. 
 It has been shown that $M''(r,r,2) = 2$ for $r$ even if the two measures are uniformly distributed among two sets $A, B$ such that $A \subseteq B$, consult \cite{fruchard2016fair}.

\medskip
In $\rr^2$ we may consider only partitions as in the proof of Theorem \ref{theorem-two-thieves} where all the vectors involved in the recursive definitions of the partition are $(1,0)$.  More precisely, every time we consider partitions of the form $\mathcal{K} *_{(v,a)} \mathcal{K}'$, it is with $v=(1,0)$.  This forces most of the hyperplanes (in this case, lines) involved in the final partition to be vertical, with possibly a line cutting the space between two consecutive vertical lines.  This implies the following corollary.

\begin{corollary}
Suppose we are given $3t-1$ measures in $\rr^2$. Then, there is a piecewise-linear path $L$ such that
\begin{itemize}
	\item $L$ uses at most $2t$ turns,
	\item $L$ is $x$-monotone,
	\item half of the segments forming $L$ are vertical (we accept segments of length zero), and
	\item $L$ splits each measure by half.
\end{itemize}
(The path $L$ may go through infinity.)
\end{corollary}

An analogous result can be obtained with $2t+1$ turns and $3t$ measures using the second part of Theorem \ref{theorem-two-thieves}.  A similar problem has been considered if only vertical and horizontal segments are allowed.  It is conjectured that $n+1$ measures can be split evenly using a path that only has $n$ turns \cite{UNO:2009wk} \cite{Karasev:2016vd}, as opposed to the $\sim \frac{3n}{2}$ measures we can split with the result above.  If the condition on the directions of the paths is completely dropped, we get an interesting problem

\begin{problem}
Given a polygonal path in the plane that uses at most $n$ turns, what is the largest number of measures that we can always guarantee to be able to split simultaneously by half?
\end{problem}

Notice again that checking the degrees of freedom does not yield the optimal number.  With $n = 1$ this would suggest that any four measures can be split, but as we are dividing by a convex angle, no more than three measures can be split simultaneously.

\begin{figure}[ht]
\centerline{\includegraphics[scale=0.7]{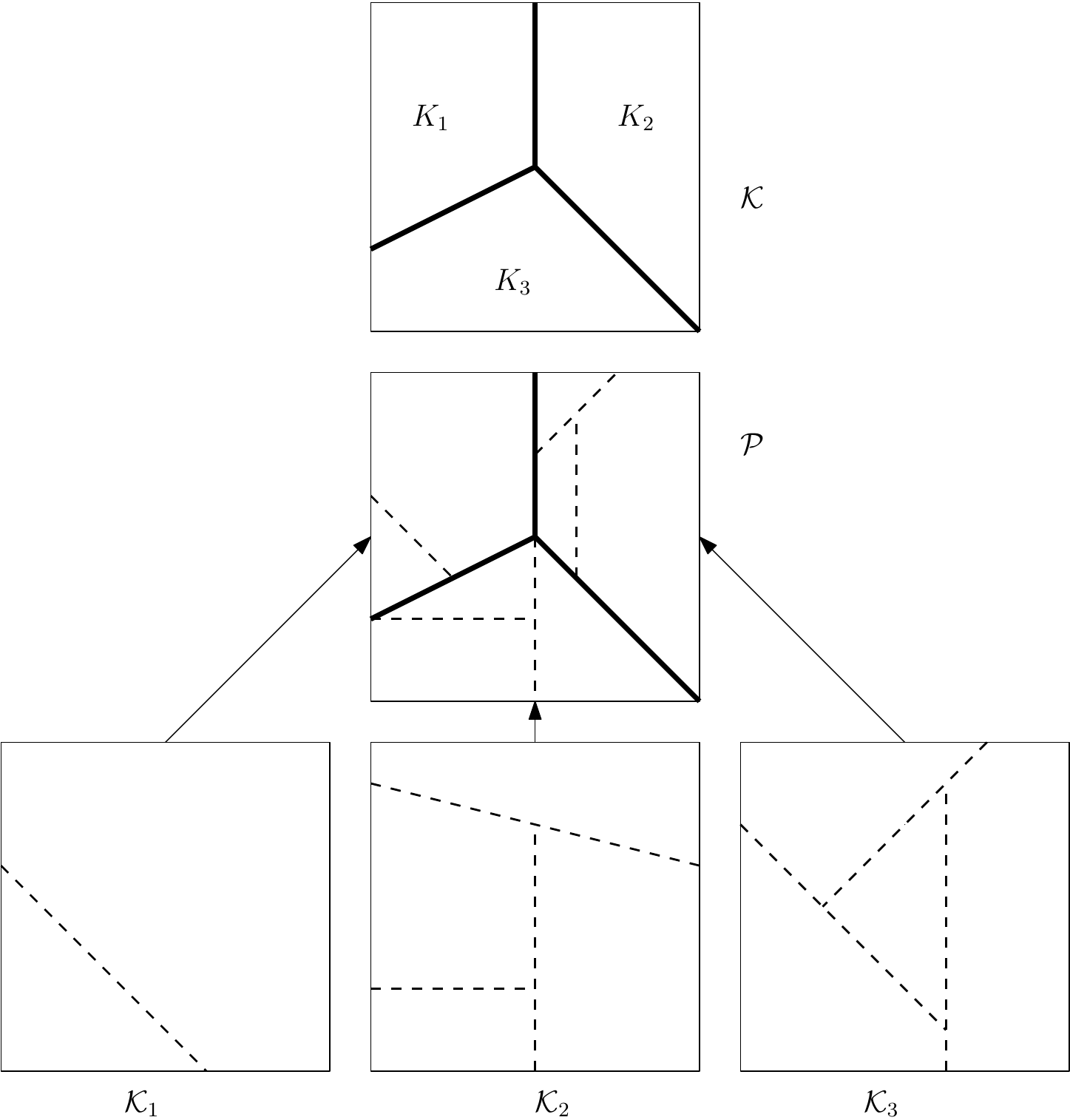}}
\caption{A partition of $\rr^2$ into eight parts formed by mixing three partitions $\K_1, \K_2, \K_3$ via a partition $\K$ with three parts.}\label{fig-window}
\end{figure}

\medskip
The construction of partitions we presented has a natural extension motivated by an analogy with the definition of the semi-direct product of groups as well as with the structural map of the little-cube operad.
For a given a partition $\K = (K_1, K_2, \ldots, K_r)$ of $\rr^d$ into $r$ parts, and further $r$ partitions $\K_1, \K_2, \ldots, \K_r$ of $\rr^d$, we can form a partition $(\K_1\times\cdots\times\K_r)\rtimes\K$ by
\[
(\K_1\times\cdots\times\K_r)\rtimes\K : = \{A \cap K_i : A \in \K_i, 1\le i \le r\}.
\]
In other words, $K_i$ is a ``window'' that lets us look at the partition $\K_i$, see Figure \ref{fig-window}.

\medskip
These configuration spaces are clearly different, and we get a richer family of partitions if we further modify $\K$.  
It is unclear how much improvement can be obtained from using this more general configuration spaces of partitions.

\medskip
\noindent
{\em Acknowledgements.} We are grateful to the referee for very careful and helpful comments on this paper.
We also thank Aleksandra Dimitrijevi\'c Blagojevi\'c for excellent observations on several drafts of this paper and many useful comments.

\bibliographystyle{amsalpha}

\newcommand{\etalchar}[1]{$^{#1}$}
\providecommand{\bysame}{\leavevmode\hbox to3em{\hrulefill}\thinspace}
\providecommand{\MR}{\relax\ifhmode\unskip\space\fi MR }
\providecommand{\MRhref}[2]{%
  \href{http://www.ams.org/mathscinet-getitem?mr=#1}{#2}
}
\providecommand{\href}[2]{#2}

\affiliationone{
   Pavle V. M. Blagojevi\'{c}\\
   Inst. Math., FU Berlin, Arnimallee 2, 14195 Berlin \\ Germany\\
   Mathematical Institute SANU, Knez Mihailova 36, 11001 Belgrade \\ Serbia
   \email{blagojevic@math.fu-berlin.de\\
   pavleb@mi.sanu.ac.rs}}
\affiliationtwo{
   Pablo Sober\'on\\
   Mathematics Department, Northeastern University, 360 Huntington Ave., 
Boston, MA 02115 \\ USA
   \email{p.soberonbravo@northeastern.edu}}

%
\end{document}